\documentclass[11pt]{amsart}
\usepackage[colorlinks]{hyperref}
\newtheorem{theorem}{Theorem}[section]

\newtheorem{cor}[theorem]{Corollary}

\theoremstyle{definition}
\newtheorem{definition}[theorem]{Definition}

\theoremstyle{remark}
\newtheorem{remark}[theorem]{Remark}

\numberwithin{equation}{section}

\begin{document}

\newcommand{\spacing}[1]{\renewcommand{\baselinestretch}{#1}\large\normalsize}
\spacing{1.14}

\title{Naturally reductive homogeneous  $(\alpha,\beta)$-metric spaces}

\author{M. Parhizkar}

\address{Department of Mathematics\\ University of Mohaghegh Ardabili\\ Ardabil\\ p.o.box. 56199-
11367-Iran.} \email{m\_parhizkar66@yahoo.com and m.parhizkar@uma.ac.ir}

\author {H. R. Salimi Moghaddam}

\address{Department of Mathematics\\ Faculty of  Sciences\\ University of Isfahan\\ Isfahan\\ 81746-73441-Iran.} \email{hr.salimi@sci.ui.ac.ir and salimi.moghaddam@gmail.com}

\keywords{Naturally reductive homogeneous space, invariant Riemannian metric, invariant $(\alpha,\beta)$-metric \\
AMS 2010 Mathematics Subject Classification: 53C60, 53C30.}


\begin{abstract}
In the present paper we study naturally reductive homogeneous $(\alpha,\beta)$-metric spaces. We show that for homogeneous $(\alpha,\beta)$-metric spaces, under a mild condition, the two definitions of naturally reductive homogeneous Finsler space, given in the literature, are equivalent. Then, we compute the flag curvature of naturally reductive homogeneous $(\alpha,\beta)$-metric spaces.
\end{abstract}

\maketitle


\section{\textbf{Introduction}}
The study of $(\alpha,\beta)-$metrics, which are introduced by M.
Matsumoto (see \cite{Ma2}.), is one of interesting and important fields in Finsler geometry. These metrics are considered not only by Finsler geometers because of their simple and interesting structure, but also by physicists because of their applications in physics. In fact the first type of $(\alpha,\beta)-$metrics, Randers metric, was introduced by G. Randers in 1941 for its application in general relativity (see \cite{Ra}). In recent years these metrics have found more applications, for example they are used by G. S. Asanov for formulation pseudo-Finsleroid gravitational field equations (see \cite{Asan1} and \cite{Asan2}.).\\
Some important examples
of $(\alpha,\beta)-$metrics are Randers metric $\alpha+\beta$,
Kropina metric $\frac{\alpha^2}{\beta}$, and Matsumoto metric
$\frac{\alpha^2}{\alpha-\beta}$, where
$\alpha(x,y)=\sqrt{\tilde{a}_{ij}(x)y^iy^j}$ and $\beta(x,y)=b_i(x)y^i$
and $\tilde{a}$ and $\beta$ are a Riemannian metric and  a 1-form respectively
as follows:
\begin{eqnarray}
  \tilde{a}&=&\tilde{a}_{ij}dx^i\otimes dx^j \\
  \beta&=&b_idx^i.
\end{eqnarray}

In this article we study naturally reductive homogeneous  $(\alpha,\beta)$-metric spaces. There are two different definitions for naturally reductive homogeneous Finsler spaces in the literature. The first one was given by S. Deng and Z. Hou in \cite{DeHo2} and the second one was given by D. Latifi in \cite{Latifi1} (also see \cite{Latifi2}). In \cite{DeHo3}, Deng and Hou showed that if a homogeneous Finsler space is naturally reductive in the sense of Latifi then it must be Berwaldian and is naturally reductive in the sense of Deng and Hou. The authors of \cite{DeHo3} pointed out that it is not clear whether there is a naturally reductive Finsler space in the sense of Deng and Hou which is not a naturally reductive Finsler space in the sense of Latifi. \\

In this paper we show that for homogeneous  $(\alpha,\beta)$-metric spaces, under a mild condition, these two definitions are equivalent.
Then we compute the flag curvature of naturally reductive homogeneous $(\alpha,\beta)$-metric spaces in the sense of Deng and Hou.\\
Now we give some preliminaries of Finsler geometry.\\
Let $M$ be a smooth $n-$dimensional manifold and $TM$ be its
tangent bundle. A Finsler metric on $M$ is a non-negative function
$F:TM\longrightarrow \Bbb{R}$ which has the following properties:
\begin{enumerate}
    \item $F$ is smooth on the slit tangent bundle
    $TM^0:=TM\setminus\{0\}$,
    \item $F(x,\lambda y)=\lambda F(x,y)$ for any $x\in M$, $y\in T_xM$ and $\lambda
    >0$,
    \item the $n\times n$ Hessian matrix $[g_{ij}(x,y)]=[\frac{1}{2}\frac{\partial^2 F^2}{\partial y^i\partial
    y^j}]$ is positive definite at every point $(x,y)\in TM^0$.
\end{enumerate}
The concept of flag curvature is an important quantity which is associated with a Finsler
manifold. This quantity is a generalization of the concept of sectional
curvature. Flag curvature is defined as follows

\begin{eqnarray}\label{flag}
  K(P,y)=\frac{g_y(R(u,y)y,u)}{g_y(y,y).g_y(u,u)-g_y^2(y,u)},
\end{eqnarray}
where $g_y(u,v)=\frac{1}{2}\frac{\partial^2}{\partial s\partial
t}(F^2(y+su+tv))|_{s=t=0}$, $P=span\{u,y\}$,
$R(u,y)y=\nabla_u\nabla_yy-\nabla_y\nabla_uy-\nabla_{[u,y]}y$ and
$\nabla$ is the Chern connection induced by $F$ (see \cite{BaChSh}
and \cite{Sh1}).\\

\begin{definition}
A Finsler space $(M, F)$ is called a Berwald space if the Chern connection
coefficients $\Gamma^i_{jk}$ in natural coordinates have no $y$ dependence.
\end{definition}
\begin{definition}
Suppose that $\alpha$ is a Riemannian metric and $\beta$ is a $1-$form as above. Let
\begin{equation}\label{alpha-norm}
    \|\beta(x)\|_\alpha:=\sqrt{\tilde{a}^{ij}(x)b_i(x)b_j(x)}.
\end{equation}
Now, let the function $F$ be defined as follows
\begin{equation}\label{alpha-beta metric}
    F:=\alpha\phi(s) \ \ \ , \ \ \ s=\frac{\beta}{\alpha},
\end{equation}
where $\phi=\phi(s)$ is a positive $C^\infty$ function on $(-b_0,b_0)$ satisfying
\begin{equation}\label{alpha-beta condition}
    \phi(s)-s\phi'(s)+(b^2-s^2)\phi''(s)>0 \ \ \ , \ \ \ |s|\leq b <b_0.
\end{equation}
Then by Lemma 1.1.2 of \cite{ChSh}, $F$ is a Finsler metric if $\|\beta(x)\|_\alpha<b_0$ for any $x\in M$.
A Finsler metric in the form (\ref{alpha-beta metric}) is called an $(\alpha,\beta)-$metric.
\end{definition}
For example if we consider $\phi(s)=1+s$, $\phi(s)=\frac{1}{1-s}$ and $\phi(s)=\frac{1}{s}$ then we have Randers, Matsumoto and Kropina metrics respectively.\\

The Riemannian metric $\tilde{a}$ induces an inner
product on any cotangent space $T^\ast_xM$ such that
$<dx^i(x),dx^j(x)>=\tilde{a}^{ij}(x)$. The induced inner product on
$T^\ast_xM$ induces a linear isometry between $T^\ast_xM$ and
$T_xM$. Then the 1-form $\beta$ corresponds to a vector field
$X$ on $M$ such that
\begin{eqnarray}
  \tilde{a}(y,X(x))=\beta(x,y).
\end{eqnarray}
Also we have $\|\beta(x)\|_{\alpha}=\|X(x)\|_{\alpha}$
(for more details see \cite{DeHo1} and \cite{Sa}.).
Therefore we can write $(\alpha,\beta)-$metrics
as follows:
\begin{eqnarray}\label{invariant alpha-beta metric}
  F(x,y)=\alpha(x,y)\phi(\frac{\tilde{a}(X(x),y)}{\alpha(x,y)}),
\end{eqnarray}
where for any $x\in M$,
$\sqrt{\tilde{a}(X(x),X(x))}=\|X(x)\|_{\alpha}<b_0$.\\
\section{\textbf{Naturally reductive homogeneous  $(\alpha,\beta)$-metric spaces}}
In order to study the geometric properties of $(\alpha,\beta)$-metrics, we need a formula for the spray coefficients of an $(\alpha,\beta)$-metric. Let $\nabla\beta=b_{i\vert j}y^{i}dx^{j}$ denote the covariant derivative of $\beta$ with respect to $\alpha$. Let $G^{i}$ and $\tilde{G}^{i}$ denote the spray coefficients of $F$ and $\alpha$, respectively, given by
\begin{eqnarray}
G^{i}=\frac{g^{il}}{4}\big{\{} [F^{2}]_{x^{k}y^{l}}y^{k} - [F^{2}]_{x^{l}}\big{\}},\qquad \tilde{G}^{i}=\frac{\tilde{a}^{il}}{4}\big{\{} [\alpha^{2}]_{x^{k}y^{l}}y^{k} - [\alpha^{2}]_{x^{l}}\big{\}},
\end{eqnarray}
where $(g^{ij})=(g_{ij})^{-1}$ and $(\tilde{a}^{ij})=(\tilde{a}_{ij})^{-1}$. Let
\begin{eqnarray}
r_{ij}:=\frac{1}{2}(b_{i\vert j}+b_{j\vert i}), \qquad s_{ij}:=\frac{1}{2}(b_{i\vert j}- b_{j\vert i}),\qquad \nonumber\\
 s^{i}_{j}:=\tilde{a}^{ik}s_{kj}, \quad s_{j}:=b_{i}s^{i}_{j}=b^{k}s_{kj}, \quad b_{ij}:=r_{ij} + b_{i}s_{j} + b_{j}s_{i},\nonumber
\end{eqnarray}
and
\begin{eqnarray}
r_{00}:=r_{ij}y^{i}y^{j}, \qquad s_{0}:= s_{i}y^{i}, \qquad s^{i}_{0}:= s^{i}_{j}y^{j}.\nonumber
\end{eqnarray}
By a direct computation, one gets the following formula:
\begin{eqnarray}
G^{i}=\tilde{G}^{i} + \alpha Qs^{i}_{0} + \Theta\Big{\{}-2Q\alpha s_{0} + r_{00} \Big{\}}\frac{y^{i}}{\alpha} + \Psi\Big{\{}-2Q\alpha s_{0} + r_{00} \Big{\}}b^{i},
\end{eqnarray}
where
\begin{eqnarray}
Q &:=&  \frac{\phi}{\phi - s\phi^{'}},\nonumber\\
\Theta &:=& \frac{(\phi - s \phi^{'})\phi^{'}}{2\Big{(}(\phi - s \phi^{'}) + (b^{2} - s^{2})\phi^{''}\Big{)}\phi} - s\Psi,\nonumber\\
\Psi &:=& \frac{\phi^{''}}{2\Big{(}(\phi - s \phi^{'}) + (b^{2} - s^{2})\phi^{''}\Big{)}},\nonumber
\end{eqnarray}
where $ s:= \frac{\beta}{\alpha} $ and $b:=\Vert\beta_{x}\Vert_{\alpha}$ (see \cite{ChSh}).\\
Thus the above, after some manipulation, becomes
\begin{eqnarray}\label{Gi}
G^{i}:=&& \dfrac{1}{2}\Gamma^{i}_{jk}y^{j}y^{k}\nonumber\\
 = &&\dfrac{1}{2}\tilde{\Gamma}^{i}_{jk}y^{j}y^{k}\nonumber\\
    && + \frac{1}{2}\alpha Qb_{j\vert k}\Big{(}\tilde{a}^{ij}y^{k} - \tilde{a}^{ik}y^{j}\Big{)}\nonumber\\
      &&+\Theta b_{j\vert k} \Big{\{}-Q\alpha(y^{j}b^{k}-y^{k}b^{j}) + y^{j}y^{k}\Big{\}}\frac{y^{i}}{\alpha}\\	           && + \Psi b_{j\vert k}\Big{\{}-Q\alpha(y^{k}b^{j}-y^{j}b^{k}) + y^{j}y^{k}\Big{\}}b^{i},\nonumber
\end{eqnarray}
where $\Gamma^{i}_{jk}$ are the Christoffel symbols of the Chern connection of $F$ and $\tilde{\Gamma}^{i}_{jk}$ are the Christoffel symbols of the Levi-Civita connection of $\alpha$.\\
Let $F = \alpha\phi(\frac{\beta}{\alpha})$ be an $(\alpha,\beta)$-metric. If $\beta$ is parallel with respect to $\alpha$ ($ b_{j\vert k} = 0 $), then by (\ref{Gi}), $\Gamma^{i}_{jk}=\tilde{\Gamma}^{i}_{jk}$. Thus $F$ is of Berwald type. The converse is also true(see \cite{Ma3} and \cite{kiku}).
\begin{definition}(see \cite{KoNu})
A homogeneous manifold $M=\frac{G}{H}$ with an invariant Riemannian metric $\tilde{a}$ is said to be naturally reductive if it admits an $Ad(H)$-invariant decomposition $\frak{g}=\frak{h}\oplus\frak{m}$ satisfying the condition
\begin{equation}\label{1}
\langle [x,y]_{\frak{m}} , z\rangle + \langle y, [x,z]_{\frak{m}}\rangle =0, \ \ \ x,y,z\in\frak{m},
\end{equation}
where $ \langle \, , \, \rangle $ is the bilinear form on $\frak{m}$ induced by $\tilde{a}$ and $[ , ]_\frak{m}$ is the projection to $\frak{m}$ with respect to the decomposition $\frak{g}=\frak{h}\oplus\frak{m}$.
\end{definition}
In particular case if we consider $H=\{e\}$ then $\frak{m}=\frak{g}$ which shows that the condition (\ref{1}) reduces to the condition
\begin{equation}\label{2}
\langle [x,y] , z\rangle + \langle y, [x,z]\rangle =0,
\end{equation}
for a bi-invariant Riemannian metric on $G$.

In literature, there are two versions of the definition of naturally reductive Finsler metrics on a manifold. The first version has been introduced by the S. Deng and Z. Hou in \cite{DeHo2} and the second one has been introduced by D. Latifi in \cite{Latifi1}.
\begin{definition}(\cite{DeHo2}, Deng and Hou)\label{def-2}
A homogeneous manifold $\frac{G}{H}$ with an invariant Finsler metric $F$ is called naturally reductive if there
exists an invariant Riemannian metric $\tilde{a}$ on $\frac{G}{H}$ such that $(\frac{G}{H}, \tilde{a})$ is naturally reductive and the connections of $\tilde{a}$ and $F$ coincide.
\end{definition}
In this definition, they assume that such a metric must be of Berwald type.
\begin{definition}(\cite{Latifi1}, Latifi)\label{def-1}
A homogeneous manifold $\frac{G}{H}$ with an invariant Finsler metric $F$ is called naturally reductive if there
exists an $Ad(H)$-invariant decomposition $\frak{g} = \frak{h} \oplus \frak{m}$ such that
\begin{equation}\label{3}
g_{y}([x,u]_{\frak{m}},v) + g_{y} (u,[x,v]_{\frak{m}}) +2C_{y}([x,y]_{\frak{m}},u,v) = 0,
\end{equation}
where $y\neq0, x, u, v \in \frak{m}$.
\end{definition}

In \cite{DeHo3}, Deng and Hou proved the following theorem.
\begin{theorem}\label{th-1}
If a homogeneous Finsler space $(\frac{G}{H}, F)$ is naturally reductive in the sense of definition \ref{def-1}, then it must be naturally reductive in the sense of definition \ref{def-2}.
\end{theorem}

\begin{cor}
Let $(\frac{G}{H},F)$ be a naturally reductive homogeneous Finsler space in the sense of definition \ref{def-1} then it must be of Berwald type.
\end{cor}

\begin{cor}
Let $(\frac{G}{H},F)$ be a homogeneous Finsler manifold, where $F$ is an invariant $(\alpha,\beta)-$metric defined by an invariant Riemannian metric $\tilde{a}$ and an invariant vector field ${X}$. If  $(\frac{G}{H}, F)$ is naturally reductive in the sense of \ref{def-1} then $(\frac{G}{H},\tilde{a})$ is naturally reductive.
\end{cor}
\begin{proof}
$(\frac{G}{H}, F)$ is naturally reductive in the sense of \ref{def-1}, therefore it is naturally reductive in the sense of \ref{def-2} and it is Berwaldian. Hence $(\frac{G}{H},\tilde{a})$ is naturally reductive.
\end{proof}

We show that, under a mild condition, two definitions of naturally reductive $(\alpha,\beta)-$metric space are equivalent.

\begin{theorem}\label{th-3}
Let $(\frac{G}{H},F)$ be a homogeneous Finsler manifold, where $F$ is an invariant $(\alpha,\beta)-$metric defined by an invariant Riemannian metric $\tilde{a}$ and an invariant vector field ${X}$ such that $\phi'(r)\neq0$, where $r:=\frac{\tilde{a}(X,y)}{\sqrt{\tilde{a}(y,y)}}=\beta(\frac{y}{\|y\|_\alpha})$. Then, the two definitions \ref{def-2} and \ref{def-1} are equivalent.
\end{theorem}

\begin{proof}
By attention to theorem \ref{th-1} it is sufficient to prove if $(\frac{G}{H},F)$ is a naturally reductive Finsler space in the sense of \ref{def-2} then it is naturally reductive in the sense of \ref{def-1}.
Suppose that $(\frac{G}{H},F)$ is a naturally reductive Finsler space in the sense of \ref{def-2}, hence it is of Berwald type and also $(\frac{G}{H}, \tilde{a})$ is naturally reductive.
We show that for all $0\neq y, z, u, v \in \frak{m}$
\begin{equation}
g_{y}([z,u]_{\frak{m}},v) + g_{y} (u,[z,v]_{\frak{m}}) +2C_{y}([z,y]_{\frak{m}},u,v) = 0.
\end{equation}
By using the formula $g_y(u,v)=\frac{1}{2}\frac{\partial^2}{\partial t \partial s}F^2(y+su+tv)|_{s=t=0}$ and some computations, for the $(\alpha,\beta)-$metric $F$ defined by relation (\ref{invariant alpha-beta metric}) we have:
\begin{eqnarray}\label{g_y}
    g_y(u,v)&=& \tilde{a}(u,v)\phi^2(r)+\tilde{a}(y,u)\phi(r)\phi'(r)\Big{(}\frac{\tilde{a}(X,v)}{\sqrt{\tilde{a}(y,y)}}-\frac{\tilde{a}(X,y)\tilde{a}(y,v)}{(\tilde{a}(y,y))^{\frac{3}{2}}}\Big{)} \nonumber\\
            && + \Big{(}(\phi'(r))^2+\phi(r)\phi''(r)\Big{)}\Big{(}\frac{\tilde{a}(X,v)}{\sqrt{\tilde{a}(y,y)}}-\frac{\tilde{a}(X,y)\tilde{a}(y,v)}{(\tilde{a}(y,y))^\frac{3}{2}}\Big{)}\nonumber\\
            && \ \ \ \ \ \ \ \ \times\Big{(}\tilde{a}(X,u)\sqrt{\tilde{a}(y,y)}-\frac{\tilde{a}(y,u)\tilde{a}(X,y)}{\sqrt{\tilde{a}(y,y)}}\Big{)}\\
            && +\frac{\phi(r)\phi'(r)}{\sqrt{\tilde{a}(y,y)}}\Big{(}\tilde{a}(X,u)\tilde{a}(y,v)-\tilde{a}(u,v)\tilde{a}(X,y)\Big{)}\nonumber,
\end{eqnarray}
where $r=\frac{\tilde{a}(X,y)}{\sqrt{\tilde{a}(y,y)}}$. \\
So for any $y\neq0, z\in\frak{m}$ we have
\begin{eqnarray}\label{eq3}
    g_{y}(y,[y,z]_\frak{m})&=&\tilde{a}(y,[y,z]_\frak{m})
    \Big{(}\phi^2(r)-\phi(r)\phi'(r)r\Big{)}\nonumber\\
    &&+\tilde{a}(X,[y,z]_\frak{m})\Big{(}\phi'(r)F(y)\Big{)}.
\end{eqnarray}
$(\frac{G}{H}, \tilde{a})$ is naturally reductive, therefore the equation (\ref{1}) shows that
\begin{equation}\label{8}
\tilde{a}(y,[y,z]_{\frak{m}}) = 0 \ \ \ , \ \ \forall y\neq0, z\in\frak{m}.
\end{equation}
By attention to the definition \ref{def-2}, $(\frac{G}{H}, F)$ and $(\frac{G}{H}, \tilde{a})$ have the same connection, the same geodesics and so the same geodesic vectors. Therefore by using Theorem 3.1 of \cite{Latifi1} for all $y\neq0, z \in \frak{m}$ we have
\begin{equation}\label{9}
g_y(y,[y,z]_{\frak{m}}) = 0.
\end{equation}
On the other hand we have $\phi'(r)\neq0$ therefore  for all $0\neq y, z \in \frak{m}$ we have
\begin{equation}\label{10}
    \tilde{a}(X,[y,z]_{\frak{m}})=0.
\end{equation}
By using the relations (\ref{g_y}), (\ref{8}), (\ref{10}) and some computations we have
\begin{eqnarray}
g_{y}([z,u]_{\frak{m}},v) &=& \tilde{a}([z,u]_{\frak{m}},v)\Big{(}\phi^{2}(r) - \phi(r)\phi'(r)r\Big{)} \nonumber\\
                            &&+ \tilde{a}([z,u]_{\frak{m}},y)\Big{(}\frac{\tilde{a}(X,v)}{\sqrt{\tilde{a}(y,y)}} -
                            \frac{\tilde{a}(y,v)}{\tilde{a}(y,y)}r\Big{)}\Big{(}\phi(r)\phi'(r) - ((\phi'(r))^{2} + \phi(r)\phi''(r))r\Big{)},
\end{eqnarray}
and
\begin{eqnarray}
g_{y}([z,v]_{\frak{m}},u) &=& \tilde{a}([z,v]_{\frak{m}},u)\Big{(}\phi^{2}(r) - \phi(r)\phi'(r)r\Big{)} \nonumber
                            \\&&+ \tilde{a}([z,v]_{\frak{m}},y)\Big{(}\frac{\tilde{a}(X,u)}{\sqrt{\tilde{a}(y,y)}} -
                            \frac{\tilde{a}(y,u)}{\tilde{a}(y,y)}r\Big{)}\Big{(}\phi(r)\phi'(r) - ((\phi'(r))^{2} + \phi(r)\phi''(r))r\Big{)}.
\end{eqnarray}
Now by using the definition
\begin{equation*}
    C_{y}(z,u,v) = \frac{1}{4}\frac{\partial}{\partial s}\frac{\partial}{\partial t}\frac{\partial}{\partial h}[F^2(y+sz+tu+hv)]\Big{\vert}_{s=t=h=0},
\end{equation*}
for Cartan tensor we have,
\begin{eqnarray}\label{C-y}
2C_{y}(u,v,z) &=& \Big{(}\frac{3\phi'(r)\phi''(r)+\phi(r)\phi'''(r)}{\tilde{a}(y,y)}\Big{)}\Big{(}\tilde{a}(X,v) -
                        \frac{\tilde{a}(v,y)\tilde{a}(X,y)}{\tilde{a}(y,y)}\Big{)}\nonumber\\
                && \times \Big{(}\tilde{a}(X,u)\sqrt{\tilde{a}(y,y)} -
                        \frac{\tilde{a}(u,y)\tilde{a}(X,y)}{\sqrt{\tilde{a}(y,y)}}\Big{)}\Big{(}\tilde{a}(X,z) - \frac{\tilde{a}(z,y)\tilde{a}(X,y)}{\tilde{a}(y,y)}\Big{)}\nonumber\\
                && + \Big{(}\frac{(\phi'(r))^{2} + \phi(r)\phi''(r)}{\tilde{a}(y,y)} \Big{)}\Big{(}\tilde{a}(X,v) -
                        \frac{\tilde{a}(v,y)\tilde{a}(X,y)}{\tilde{a}(y,y)}\Big{)}\nonumber\\
                &&\times \Big{(}\tilde{a}(X,u)\tilde{a}(z,y) - \tilde{a}(u,z)\tilde{a}(X,y) - \tilde{a}(u,y)\tilde{a}(z,X)
                        +\frac{\tilde{a}(z,y)\tilde{a}(u,y)\tilde{a}(X,y)}{\tilde{a}(y,y)}\Big{)}\nonumber\\
                &&-\Big{(}\dfrac{(\phi'(r))^{2}+\phi(r)\phi''(r)}{\tilde{a}(y,y)\sqrt{\tilde{a}(y,y)}}\Big{)}
                        \Big{(}\tilde{a}(X,u)\sqrt{\tilde{a}(y,y)}-\frac{\tilde{a}(u,y)\tilde{a}(X,y)}{\sqrt{\tilde{a}(y,y)}}\Big{)}\nonumber\\
                && \times \Big{(}\tilde{a}(z,y)\tilde{a}(X,v) + \tilde{a}(v,z)\tilde{a}(X,y) + \tilde{a}(v,y)\tilde{a}(X,z) -
                        \frac{3\tilde{a}(v,y)\tilde{a}(X,y)\tilde{a}(z,y)}{\tilde{a}(y,y)}\Big{)}\\
                && + \Big{(}\frac{(\phi'(r))^{2} + \phi(r)\phi''(r)}{\tilde{a}(y,y)}\Big{)}\Big{(}\tilde{a}(X,z) -
                        \frac{\tilde{a}(z,y)\tilde{a}(X,y)}{\tilde{a}(y,y)}\Big{)}\nonumber\\
                && \times \Big{(} \tilde{a}(X,u)\tilde{a}(v,y) - \tilde{a}(u,v)a(X,y) + \tilde{a}(u,y)\tilde{a}(X,v) -
                        \frac{\tilde{a}(v,y)\tilde{a}(u,y)\tilde{a}(X,y)}{\tilde{a}(y,y)}\Big{)}\nonumber\\
                && + \frac{\phi(r)\phi'(r)}{\sqrt{\tilde{a}(y,y)}} \Big{(}\tilde{a}(X,u)\tilde{a}(v,z) + \tilde{a}(u,v)\tilde{a}(X,z)
                        +\tilde{a}(u,z)\tilde{a}(X,v)\nonumber \\
                && - \frac{\tilde{a}(z,y)\tilde{a}(v,y)\tilde{a}(X,u) + \tilde{a}(u,v)\tilde{a}(X,y)\tilde{a}(z,y) +
                        \tilde{a}(u,y)\tilde{a}(X,v)\tilde{a}(z,y)}{\tilde{a}(y,y)}\nonumber\\
                && - \frac{\tilde{a}(v,z)\tilde{a}(u,y)\tilde{a}(X,y) + \tilde{a}(v,y)\tilde{a}(u,z)\tilde{a}(X,y) +
                        \tilde{a}(v,y)\tilde{a}(u,y)\tilde{a}(X,z)}{\tilde{a}(y,y)}\nonumber\\
                && + \frac{3\tilde{a}(z,y)\tilde{a}(v,y)\tilde{a}(u,y)\tilde{a}(X,y)}{(\tilde{a}(y,y))^{2}} \Big{)}.\nonumber
\end{eqnarray}
So we have
\begin{eqnarray}
2C_{y}([z,y]_{\frak{m}},u,v) &=& \tilde{a}([z,y]_{\frak{m}},u)\Big{(}\frac{\tilde{a}(X,v)}{\sqrt{\tilde{a}(y,y)}}
                                    - \frac{\tilde{a}(v,y)}{\tilde{a}(y,y)}r\Big{)}\Big{(}\phi(r)\phi'(r) - \big{(}(\phi'(r))^{2} + \phi(r)\phi''(r)\big{)}r\Big{)}\nonumber \\
                                && + \tilde{a}([z,y]_{\frak{m}},v)\Big{(}\frac{\tilde{a}(X,u)}{\sqrt{\tilde{a}(y,y)}} -
                                    \frac{\tilde{a}(u,y)}{\tilde{a}(y,y)}r\Big{)}\Big{(}\phi(r)\phi'(r) - \big{(}(\phi'(r))^{2} +
                                    \phi(r)\phi''(r)\big{)}r\Big{)}.
\end{eqnarray}
Therefore
\begin{eqnarray}\label{eq14}
  g_{y}([z,u]_{\frak{m}},v) + g_{y} (u,[z,v]_{\frak{m}}) + 2C_{y}([z,y]_{\frak{m}},u,v) &=& \nonumber\\
            &&\hspace*{-5cm}\big{(} \tilde{a}([z,u]_{\frak{m}},v) + \tilde{a}([z,v]_{\frak{m}},u)\big{)}\Big{(}(\phi(r))^{2} - \phi(r)\phi'(r)r\Big{)}\nonumber\\
            &&\hspace*{-5cm}+ \big{(}\tilde{a}([z,u]_{\frak{m}},y) + \tilde{a}([z,y]_{\frak{m}},u)\big{)}\Big{(}\dfrac{\tilde{a}(X,v)}{\sqrt{\tilde{a}(y,y)}}-
                    \dfrac{\tilde{a}(v,y)}{\tilde{a}(y,y)}r\Big{)}\nonumber\\
            &&\hspace*{-5cm}\times\Big{(}\phi(r)\phi'(r) - \big{(}(\phi'(r))^{2} + \phi(r)\phi''(r)\big{)}r\Big{)}\\
            &&\hspace*{-5cm}+ \big{(}\tilde{a}([z,v]_{\frak{m}},y) + \tilde{a}([z,y]_{\frak{m}},v)\big{)}\Big{(}\dfrac{\tilde{a}(X,u)}{\sqrt{\tilde{a}(y,y)}} -
                    \dfrac{\tilde{a}(u,y)}{\tilde{a}(y,y)}r\Big{)}\nonumber\\
            &&\hspace*{-5cm}\times\Big{(}\phi(r)\phi'(r) - \big{(}(\phi'(r))^{2} + \phi(r)\phi''(r)\big{)}r\Big{)}=0,\nonumber
\end{eqnarray}
for all $y \neq 0, u, v, z \in \frak{m}$.
\end{proof}

\begin{remark}
The condition $\phi'(r)\neq0$ which is assumed in the previous theorem is not a restrictive condition. For example the famous $(\alpha,\beta)$-metrics Randers, Kropina and Matsumoto are satisfy this condition because for any $y \neq 0\in \frak{m}$ we have:
\begin{description}
  \item[Randers] $\phi'(r)=1\neq0$
  \item[Kropina] $\phi'(r)=-\frac{\tilde{a}(y,y)}{\tilde{a}(X,y)}=-F(y)\neq0$
  \item[Matsumoto] $\phi'(r)=\frac{\tilde{a}(y,y)}{(\sqrt{\tilde{a}(y,y)}-\tilde{a}(X,y))^2}\neq0$.
\end{description}

\end{remark}

\begin{theorem}\label{th-5}
Suppose that  $(\frac{G}{H},F)$ is a homogeneous Finsler manifold, where $F$ is an invariant $(\alpha,\beta)-$metric defined by an invariant Riemannian metric $\tilde{a}$ and an invariant vector field ${X}$ such that $\phi'(r)\neq0$.
Then the homogeneous Finsler manifold $(\frac{G}{H},F)$ is naturally reductive in the sense of \ref{def-2} (or equivalently \ref{def-1}) if and only if $(ad_{x})_{\frak{m}}$, for every $x\in \frak{g}$, is skew-adjoint with respect to $\tilde{a}$  and $\tilde{a}(X, [\frak{m},\frak{m}]_{\frak{m}}) = 0$.
\end{theorem}
 \begin{proof}
   Suppose that, for every $x\in \frak{g}$, $(ad_{x})_{\frak{m}}$ is skew-adjoint and $\tilde{a}(X, [\frak{m},\frak{m}]_{\frak{m}}) = 0$. Therefore by attention to the proof of theorem \ref{th-3}, the equations (\ref{8}), (\ref{10}) and therefore (\ref{eq14}) hold.\\
 Conversely Let $(\frac{G}{H},F)$ be naturally reductive. By using theorem \ref{th-3}, for any $0\neq y, z\in\frak{m}$ we have
\begin{equation}\label{4}
    g_{y}(y,[y,z]_\frak{m})=0,
\end{equation}
and
\begin{equation}\label{5}
    \tilde{a}(y,[y,z]_{\frak{m}}) = 0.
\end{equation}
Equation (\ref{5}) shows that $(ad_{x})_{\frak{m}}$ is skew-adjoint, for every $x\in \frak{g}$.\\
Now equations (\ref{4}), (\ref{5}) and (\ref{eq3}) together with the condition $\phi'(r)\neq0$ show that
\begin{equation}\label{7}
    \tilde{a}(X, [\frak{m},\frak{m}]_{\frak{m}}) = 0.
\end{equation}
\end{proof}
\begin{cor}\label{Cor3}
Let $G$ be a connected Lie group with a left invariant Finsler metric $F$,  where $F$ is an  $(\alpha,\beta)-$metric defined by a left invariant Riemannian metric $\tilde{a}$ and a left invariant vector field $X$ such that $\phi'(r)\neq0$. Then the $(\alpha,\beta)-$ metric $F$ is bi-invariant if and only if $ad_{x}$ is skew-adjoint with respect to $\tilde{a}$ for every $ x\in \frak{g}$, and $\tilde{a}(X, [\frak{g}, \frak{g}]) = 0$.
\end{cor}
\begin{proof}
It is sufficient, in theorem \ref{th-5}, consider $H=\{e\}$ and use Theorem 2.3 of \cite{DeHo2}. \\
\end{proof}

Now we give the formula of the flag curvature of naturally reductive $(\alpha,\beta)-$metric spaces, in the sense of \ref{def-2}, explicitly.

\begin{theorem}\label{flag curvature1}
Let $(\frac{G}{H},F)$ be a naturally reductive homogeneous Finsler manifold in the sense of definition \ref{def-2}, where $F$ is an invariant $(\alpha,\beta)-$metric defined by an invariant Riemannian metric $\tilde{a}$ and an invariant vector field ${X}$. Let $\{P,y\}$ be a flag constructed in $(e,y)$, and $\{u,y\}$ be an orthonormal basis of $P$ with respect to the inner product induced by the Riemnnian metric $\tilde{a}$ on $\frak{m}$. Then the flag curvature of the flag $\{P,y\}$ in $\frak{m}$ is given by
\begin{eqnarray}\label{eq6}
  K(P,y)&=&\frac{1}{\theta}\Big{\{}\Big{(}\phi^2(r)-\phi(r)\phi'(r)r\Big{)}\big{(}\frac{1}{4}\tilde{a}([y,[u,y]_{\frak{m}}]_{\frak{m}},u)
                    +\tilde{a}([y,[u,y]_{\frak{h}}],u)\big{)}\nonumber\\
        &&+\Big{(}(\phi')^2(r)+\phi(r)\phi''(r)\Big{)}\tilde{a}(X,u)\big{(}\frac{1}{4}\tilde{a}([y,[u,y]_{\frak{m}}]_{\frak{m}},X) + \tilde{a}(
                    [y,[u,y]_{\frak{h}}],X)\big{)}\Big{\}},
\end{eqnarray}
where $r=\frac{{\tilde{a}}(X,y)}{{\tilde{a}}(y,y)}={\tilde{a}}(X,y)$ and $\theta=\phi^2(r)\Big{(}\phi^2(r)+\phi(r)\phi''(r)\tilde{a}^2(X,u)-\phi(r)\phi'(r)r\Big{)}$.
\end{theorem}
\begin{proof}
Since the $(\alpha,\beta)$-metric $F$ is naturally reductive in the sense of definition \ref{def-2} therefore there exists an invariant Riemannian metric $g$ on $\frac{G}{H}$ such that $(\frac{G}{H}, g)$ is naturally reductive and the Levi-Civita connection of $g$ and the Chern connection of $F$ and therefore their curvature tensors coincide. So we have (see \cite{KoNu}.)
\begin{equation}\label{curvature tensor}
    R(u,y)y=\frac{1}{4}[y,[u,y]_{\frak{m}}]_{\frak{m}} + [y,[u,y]_{\frak{h}}] \quad,\quad \forall u, y \in \frak{m}.
\end{equation}
Now by using the relations (\ref{curvature tensor}) and (\ref{g_y}) we have:
\begin{eqnarray}\label{4g_y(R,u)}
  g_y(R(u,y)y,u)&=&\Big{(}\phi^2(r)-\phi(r)\phi'(r)r\Big{)}\big{(}\frac{1}{4}\tilde{a}([y,[u,y]_{\frak{m}}]_{\frak{m}},u)
                            +\tilde{a}([y,[u,y]_{\frak{h}}],u)\big{)}\nonumber\\
                && +\Big{(}\phi(r)\phi'(r)\tilde{a}(X,u)-\big{(}(\phi')^2(r)+\phi(r)\phi''(r)\big{)}\tilde{a}(X,u)r\Big{)}\nonumber \\
                &&\times \big{(}\frac{1}{4}\tilde{a}([y,[u,y]_{\frak{m}}]_{\frak{m}},y) + \tilde{a}( [y,[u,y]_{\frak{h}}],y)\big{)}\\
                &&+\Big{(}(\phi')^2(r)+\phi(r)\phi''(r)\Big{)}\tilde{a}(X,u)\big{(}\frac{1}{4}\tilde{a}([y,[u,y]_{\frak{m}}]_{\frak{m}},X) + \tilde{a}( [y,[u,y]_{\frak{h}}],X)\big{)},\nonumber
\end{eqnarray}
\begin{equation}\label{g_y(u,u)}
    g_y(u,u)=\phi^2(r)+\Big{(}(\phi')^2(r)+\phi(r)\phi''(r)\Big{)}\tilde{a}^2(X,u)-\phi(r)\phi'(r)r,
\end{equation}
\begin{equation}\label{g_y(y,y)}
    g_y(y,y)=\phi^2(r),
\end{equation}
and
\begin{equation}\label{g_y(y,u)}
    g_y(y,u)=\phi(r)\phi'(r)\tilde{a}(X,u).
\end{equation}
We know that for any Riemannian manifold $(M,g)$ we have $g(R(X_1,X_2)X_3,X_4)=-g(R(X_1,X_2)X_4,X_3)$ so by using equation (\ref{curvature tensor}) we have
\begin{equation}
   \tilde{a}(R(u,y)y,y)=\frac{1}{4}\tilde{a}([y,[u,y]_{\frak{m}}]_{\frak{m}},y) + \tilde{a}( [y,[u,y]_{\frak{h}}],y)=0.
\end{equation}
Now it is sufficient to substitute the equations (\ref{4g_y(R,u)}), (\ref{g_y(u,u)}), (\ref{g_y(y,y)}) and (\ref{g_y(y,u)}) in (\ref{flag}).
\end{proof}
\begin{remark}
In theorem \ref{flag curvature1} if we let $(\frac{G}{H},F)$ be a naturally reductive homogeneous Finsler manifold in the sense of definition \ref{def-1} then the results are true by theorem \ref{th-1}.
\end{remark}

\bibliographystyle{amsplain}

\end{document}